\documentclass{amsart}
\usepackage{amssymb, amsmath, latexsym, mathabx, dsfont,tikz}
\usepackage{multirow}
\usepackage{setspace}
\usepackage{enumerate}
\usepackage{diagbox}
\usepackage{stmaryrd}
\usepackage{makecell}

\renewcommand{\baselinestretch}{\baselinestretch}
\renewcommand{\baselinestretch}{1.1}
\numberwithin{equation}{section}

\newtheorem{thm}{Theorem}[section]
\newtheorem{lem}[thm]{Lemma}

\newtheorem{prop}[thm]{Proposition}

\theoremstyle{definition}

\theoremstyle{remark}

\numberwithin{equation}{section}

\newcommand{\ra}{{\ \longrightarrow \ }}

\newcommand{\rat}{{\ \overset{3}{\longrightarrow} \ }}

\newcommand{\gen}{\text{gen}}
\newcommand{\spn}{\text{spn}}

\newcommand{\n}{{\mathbb N}}
\newcommand{\z}{{\mathbb Z}}

\newcommand{\Mod}[1]{\ (\mathrm{mod}\ #1 )}

\begin{document}
\title[]{Tight universal octagonal forms}

\author{Jangwon Ju and Mingyu Kim}

\address{Department of Mathematics, University of Ulsan, Ulsan 44610, Korea}
\email{jangwonju@ulsan.ac.kr}

\address{Department of Mathematics, Sungkyunkwan University, Suwon 16419, Korea}
\email{kmg2562@skku.edu}

\thanks{This research of the first author was supported by the National Research Foundation of Korea(NRF) grant funded by the Korea government(MSIT) (NRF-2019R1F1A1064037)}

\thanks{This research of the second author was supported by the National Research Foundation of Korea(NRF) grant funded by the Korea government(MSIT) (NRF-2021R1C1C2010133)}

\begin{abstract}
Let $P_8(x)=3x^2-2x$. For positive integers $a_1,a_2,\dots,a_k$, a polynomial of the form $a_1P_8(x_1)+a_2P_8(x_2)+\cdots+a_kP_8(x_k)$ is called an octagonal form.
For a positive integer $n$, an octagonal form is called tight $\mathcal T(n)$-universal if it represents (over $\z$) every positive integer greater than or equal to $n$ and does not represent any positive integer less than $n$.
In this article, we find all tight $\mathcal T(n)$-universal octagonal forms for every $n\ge 2$.
Furthermore, we provide an effective criterion on tight $\mathcal T(n)$-universality of an arbirary octagonal form, which is a generalization of ``15-Theorem" of Conway and Schneeberger.
\end{abstract}
\maketitle

\section{Introduction}
A (positive definite integral) quadratic form $f$ is called universal if it represents all positive integers.
Ramanujan \cite{R} determined all diagonal quaternary universal quadratic forms (see also \cite{D}).
In 1993, Conway and Schneeberger announced the so called ``15-Theorem" which states that a quadratic form representing all positive integers up to 15 actually represents every positive integer.
Bhargava \cite{B} found a simple proof for the fifteen theorem and this theorem was generalized by Bhargava and Hanke \cite{BHa} to the case of non-classic integral quadratic forms.

For an integer $m\ge 3$, define a polynomial
$$
P_m(x)=\frac{(m-2)x^2-(m-4)x}{2}.
$$
An integer of the form $P_m(u)$ for some integer $u$ is called a generalized $m$-gonal number.
A polynomial of the form $a_1P_m(x_1)+a_2P_m(x_2)+\cdots+a_kP_m(x_k)$ for positive integers $a_1,a_2,\dots,a_k$ is called {\it a sum of generalized $m$-gonal numbers} or {\it a $k$-ary $m$-gonal form}.
We say that a nonnegative integer $\nu$ is represented by an $m$-gonal form $a_1P_m(x_1)+a_2P_m(x_2)+\cdots+a_kP_m(x_k)$ if the Diophantine equation
$$
a_1P_m(x_1)+a_2P_m(x_2)+\cdots+a_kP_m(x_k)=\nu
$$
has an integer solution.
In 1862, Liouville determined all ternary triangular forms representing every positive integer.
Bosma and Kane \cite{BK} proved the triangular theorem of eight which states that if a triangular form $g$ represents the positive integers 1,2,4,5 and 8, then $g$ represents all positive integers.
Oh and the first author \cite{JO} classified universal octagonal forms and proved ``60-Theorem" analogous to the above theorems.
The pentagonal case was resolved by the first author \cite{Ju}.

Recently, Oh and the second author \cite{KO} studied tight universal quadratic forms.
For a positive integer $n$, an integral quadratic form $f$ is called {\it tight $\mathcal T(n)$-universal} if the set of all nonzero integers represented by the quadratic form $f$ is equal to the set of all integers greater than $n$.
For $n\ge 2$, the notion of tight $\mathcal T(n)$-universality is a particular form of almost universality of quadratic forms, and clearly, the tight $\mathcal T(1)$-universality coincides with the universality.
In \cite{KO}, all diagonal tight $\mathcal T(n)$-universal quadratic forms are determined for every positive integer $n$ greater than 1.
Note that one might easily classify diagonal universal quadratic forms by using Conway-Schneeberger fifteen theorem.
The second author \cite{MG} generalized the notion of tight universality to an $m$-gonal form in an expected way and determined all tight $\mathcal T(n)$-universal $m$-gonal forms for the following pairs $(m,n)$;
$$
\text{(i)}\ m=3,\ n\ge 3;\ \text{(ii)}\ m=5,\ n\ge 7;\ \text{(iii)}\ m=7,\ n\ge 11;\ \text{(iv)}\ m\ge 8,\ n\ge 2m-5.
$$
Note that the case of $(m,n)=(3,2)$ was dealt with in \cite{Ju2}.
Thus the problem of classifying all tight $\mathcal T(n)$-universal $m$-gonal forms have already been resolved for the cases when $m=3,4$.

The aim of this paper is to find all tight $\mathcal T(n)$-universal octagonal forms for every $n\ge 2$.
This finishes the classification of tight universal octagonal forms since universal octagonal forms have already been determined.
In Section 2, general notation and terminologies will be given.
We also introduce an algorithm giving all tight $\mathcal T(n)$-universal $m$-gonal forms for a given pair $(m,n)$.
In Section 3, we prove several lemmas and propositions on representations by octagonal forms.
Finally, the tight $\mathcal T(n)$-universal octagonal forms are classified for every $n\ge 2$ in Section 4.
The main results of this article are the following four theorems which give complete classification of tight universal octagonal forms.

\begin{thm} \label{thm2}
There are exactly 57 new tight $\mathcal T(2)$-universal octagonal forms, which are listed in Table \ref{tablet2}.
Furthermore, an octagonal form $g$ is tight $\mathcal T(2)$-universal if and only if $g$ does not represent 1 and does represent
$$
2,3,4,6,8,9,11,12,14\ \ \text{and}\ \ 18.
$$
\end{thm}

\begin{thm} \label{thm3}
There are exactly 147 new tight $\mathcal T(3)$-universal octagonal forms, which are listed in Table \ref{tablet3}.
Furthermore, an octagonal form $g$ is $\mathcal T(3)$-universal if and only if $g$ does not represent 1,2 and does represent
$$
3,4,5,6,13,14,16,17,21,22,27\ \ \text{and}\ \ 36.
$$
\end{thm}

\begin{thm} \label{thm4}
There are exactly 22 new tight $\mathcal T(4)$-universal octagonal forms, which are listed in Table \ref{tablet4}.
Furthermore, an octagonal form $g$ is $\mathcal T(4)$-universal if and only if $g$ does not represent 1,2,3 and does represent
$$
4,5,6,7,8,23\ \ \text{and}\ \ 28.
$$
\end{thm}

\begin{thm} \label{thm5}
Let $n$ be an integer greater than or equal to 5.
There are exactly two new tight $\mathcal T(n)$-universal octagonal forms
$$
p_8(n,n,n+1,n+2,\dots,2n-1)\ \ \text{and}\ \ p_8(n,n+1,n+2,\dots,2n-1,2n).
$$
Furthermore, an octagonal form $g$ is tight $\mathcal T(n)$-universal if and only if $g$ does not represent any positive integer less than $n$ and does represent
$$
n,n+1,n+2,\dots,2n-1\ \ \text{and}\ \ 2n.
$$
\end{thm}

\section{Preliminaries}
For $k=1,2,3,\dots$, we define a set $\mathcal{N}(k)$ by
$$
\mathcal{N}(k)=\{ \mathbf{a}=(a_1,a_2,\dots,a_k)\in \n^k : a_1\le a_2\le \cdots \le a_k\},
$$
and put $\mathcal{N}=\bigcup_{k=1}^{\infty}\mathcal{N}(k)$.
For $\mathbf{a}=(a_1,\dots,a_k)\in \mathcal{N}(k)$ and $\mathbf{b}=(b_1,\dots,b_s)\in \mathcal{N}(s)$ with $k\le s$, we write $\mathbf{a}\preceq \mathbf{b}\ (\mathbf{a} \prec \mathbf{b})$ if $(a_i)_{1\le i\le k}$ is a (proper) subsequence of $(b_j)_{1\le j\le s}$.
Given a vector $\mathbf{a}\in \mathcal{N}(k)$ and a positive integer $a$, we define a vector $\mathbf{a}*a$ by
$$
\mathbf{a}*a=(a_1,a_2,\dots,a_i,a,a_{i+1},a_{i+2},\dots,a_k)\in \mathcal{N}(k+1),
$$
where $i$ is the maximum index satisfying $a_i\le a$, i.e., $\mathbf{a}*a$ is the vector having coefficients $a_1,a_2,\dots,a_k$ and $a$ in ascending order so that $\mathbf{a}*a\in \mathcal{N}(k+1)$.
For $\mathbf{a}\in \mathcal{N}(k)$ and $\mathbf{b}=(b_1,b_2,\dots,b_s)\in \mathcal{N}(s)$ with $k\le s$, define $\mathbf{a}*\mathbf{b}$ to be the vector
$$
\mathbf{a}*b_1*b_2*\cdots*b_s \in \mathcal{N}(k+s).
$$
We identify $\mathcal{N}(1)$ with $\n$ so that $3*7*2*5$ denotes the vector $(2,3,5,7)\in \mathcal{N}(4)$.

For a vector $\mathbf{a}=(a_1,a_2,\dots,a_k)\in \mathcal{N}(k)$, an octagonal form
$$
a_1P_8(x_1)+a_2P_8(x_2)+\cdots+a_kP_8(x_k)
$$
will be denoted by $p_8(\mathbf{a})$ or $p_8(a_1,a_2,\dots,a_k)$. We define
$$
R(\mathbf{a})=\{ a_1P_8(u_1)+a_2P_8(u_2)+\cdots+a_kP_8(u_k) : u_i\in \z \} \ \ \text{and}\ \ R'(\mathbf{a})=R(\mathbf{a})-\{0\}.
$$
If a nonnegative integer $\nu$ is in the set $R(\mathbf{a})$, then we say that $\nu$ is represented by the octagonal form $p_8(\mathbf{a})$ and write $\nu \ra p_8(\mathbf{a})$.
Let $n$ be a positive integer.
An octagonal form $p_8(\mathbf{a})$ is called {\it $\mathcal T(n)$-universal} if $\mathcal T(n) \subseteq R'(\mathbf{a})$, and {\it tight $\mathcal T(n)$-universal} if $R'(\mathbf{a})=\mathcal T(n)$.
A tight $\mathcal T(n)$-universal octagonal form $p_8(\mathbf{a})$ is called {\it new} if $R'(\mathbf{b})\subsetneq \mathcal T(n)$ for every vector $\mathbf{b}\in \mathcal{N}$ satisfying $\mathbf{b}\prec \mathbf{a}$.

Now, we introduce an algorithm in \cite{JK} that gives all new tight $\mathcal T(n)$-universal $m$-gonal forms for a given pair $(m,n)$.
Since we only consider the case when $m=8$, we describe the algorithm restricted to this case.
Let $n$ be a positive integer.
For $\mathbf{a}\in \mathcal{N}$, we define a set $\Psi(\mathbf{a})$ by
$$
\Psi(\mathbf{a})=\mathcal T(n)-R'(\mathbf{a}),
$$
and a function $\psi : \mathcal{N}\to \mathcal T(n)\cup \{ \infty \}$ by
$$
\psi(\mathbf{a})=\begin{cases}\min(\Psi(\mathbf{a}))&\text{if}\ \ \Psi(\mathbf{a})\neq \emptyset,\\
\infty&\text{otherwise.}\end{cases}
$$
For a vector $\mathbf{a}$ with $\psi(\mathbf{a})<\infty$, we define a set $\mathcal E(\mathbf{a})$ to be
$$
\mathcal E(\mathbf{a})=\{ g\in \z : n\le g\le \psi(\mathbf{a})-n\} \cup \{ \psi(\mathbf{a})\}.
$$
Note that if $\psi(\mathbf{a})<2n$, then $\mathcal E(\mathbf{a})=\{ \psi(\mathbf{a})\}$.
For $k=1,2,3,\dots$, we define subsets $E(k),U(k),NU(k)$ and $A(k)$ of $\mathcal{N}(k)$ recursively as follow;
Put $E(1)=\{(n)\}$.
Define 
$$
U(k)=\{ \mathbf{a}\in E(k) : \psi(\mathbf{a})=\infty \}.
$$
Let $NU(k)$ to be the set of all vectors $\mathbf{a}$ in $U(k)$ such that $\mathbf{b}\not\in \bigcup_{i=1}^{k-1}U(i)$ for every $\mathbf{b}\in \mathcal{N}$ satisfying $\mathbf{b}\prec \mathbf{a}$.
Let $A(k)=E(k)-U(k)$ and
$$
E(k+1)=\bigcup_{\mathbf{a}\in A(k)}\left\{ \mathbf{a}*g : g\in \mathcal E(\mathbf{a})\right\}.
$$
This algorithm terminates once $A(k)=\emptyset$.
Note that for $\mathbf{a}\in \mathcal{N}(k)$, a $k$-ary $m$-gonal form $(\mathcal{GP}_m,\mathbf{a})$ is new tight $\mathcal T(n)$-universal if and only if $\mathbf{a}\in NU(k)$.

\begin{lem} \label{lemesc}
Under the notations given above, there is a positive integer $l=l(n)$ depending on $n$ such that $A(l)=\emptyset$.
\end{lem}

\begin{proof}
See \cite[Lemma 3.2]{JK}.
\end{proof}

Let $l=l(n)$ denote the smallest positive integer satisfying $A(l)=\emptyset$.
Define a set $C=C(n)$ to be
$$
C(n)=\{n\} \cup \bigcup_{k=1}^{l-1}\{ \psi(\mathbf{a}) : \mathbf{a}\in A(k)\}.
$$

\begin{lem} \label{lemcv}
For $\mathbf{a}\in \mathcal{N}$, under the notations given above, $R'(\mathbf{a})=\mathcal T(n)$ if and only if $R'(\mathbf{a})\cap \{1,2,\dots,n-1\}=\emptyset$ and $C(n)\subset R'(\mathbf{a})$.
\end{lem}

\begin{proof}
See \cite[Theorem 3.3]{JK}.
\end{proof}

\begin{lem} \label{lembase}
Let $n$ be an integer greater than 1.
Under the notations given above, we have the following;
\begin{enumerate} [(i)]
\item $\{n,n+1,n+2,\dots,2n\} \subseteq C(n)$;
\item $E(k)=\{(n,n+1,n+2,\dots,n+k-1)$ for $k=1,2,\dots,n$;
\item $U(k)=\emptyset$ for $k=1,2,\dots,n$;
\item $A(k)=E(k)$ for $k=1,2,\dots,n$;
\item $E(n+1)=\{(n,n,n+1,n+2,\dots,2n-1),(n,n+1,n+2,\dots,2n-1,2n)\}$.
\end{enumerate}
\end{lem}

\begin{proof}
See \cite[Proposition 3.5]{JK}.
\end{proof}

We adopt the geometric language of quadratic spaces and lattices, generally following \cite{OM}.
Let $R$ be the ring of rational integers $\z$ or the ring of $p$-adic integers $\z_p$ for a prime $p$, and let $F$ be the field of fractions of $R$.
An $R$-lattice $L=\z \mathbf{v}_1+\z \mathbf{v}_2+\cdots+\z \mathbf{v}_k$ is a free $R$-submodule of a quadratic space $(W,Q)$, where $W$ is a $k$-dimensional vector space over $F$ and $Q$ is a quadratic map from $W$ to $F$.
The symmetric bilinear map $B:W\times W\to F$ associated to $Q$ is given by
$$
B(\mathbf{x},\mathbf{y})=\frac{1}{2}\{Q(\mathbf{x}+\mathbf{y})-Q(\mathbf{x})-Q(\mathbf{y})\}.
$$
The matrix $B(\mathbf{v}_i,\mathbf{v}_j)_{1\le i,j\le k}$, denoted $M_L$, is called the Gram matrix of $L$ in the basis $\{ \mathbf{v}_1,\mathbf{v}_2,\dots,\mathbf{v}_k\}$.
Throughout the article, we always assume that a $\z$-lattice $L$ is positive definite in the sense that a Gram matrix $M_L$ of $L$ is positive definite.
The corresponding quadratic form $f_L$ is defined by
$$
f_L(x_1,x_2,\dots,x_k)=\sum_{1\le i,j\le k} B(\mathbf{v}_i,\mathbf{v}_j)x_ix_j.
$$
We abuse the notation and write $L=A$, when $A$ is the Gram matrix of $L$ in a $\z$-basis for $L$.
The genus of $L$ (spinor genus of $L$) will be denoted by $\gen(L)$ ($\spn(L)$, respectively).
We abuse the notation when we explicitly write down the isometry classes in the (spinor) genus of $L$ and write, e.g.,
$$
\gen(L)=\{L,M\},
$$
when the class number of $L$ is two and the genus mate is $M$.
A diagonal $k\times k$-matrix $A$ with diagonal entries $a_1,a_2,\dots,a_k$ will be simply denoted by $\langle a_1,a_2,\dots,a_k\rangle$.
Both a $\z$-lattice $L$ having $A$ as its Gram matrix and the corresponding diagonal quadratic form $a_1x_1^2+a_2x_2^2+\cdots+a_kx_k^2$ will also be denoted by $\langle a_1,a_2,\dots,a_k\rangle$.
Let $\nu$ be a nonnegative integer.
For a $\z$-lattice $L$, we write $\nu \ra L$ when $\nu$ is represented by $L$.
For a diagonal $\z$-lattice $\langle b_1,b_2,\dots,b_k\rangle$, we write
$$
\nu \rat \langle b_1,b_2,\dots,b_k\rangle
$$
if there is a vector $(y_1,y_2,\dots,y_k)\in \z^k$ with $(y_1y_2\cdots y_k,3)=1$ such that
$$
b_1y_1^2+b_2y_2^2+\cdots+b_ky_k^2=\nu.
$$
Under these notations, for a nonnegative integer $u$, one may easily see that
$$
u\ra p_8(a_1,a_2,\dots,a_k)\ \ \text{if and only if}\ \ 3u+a_1+a_2+\cdots+a_k\rat \langle a_1,a_2,\dots,a_k\rangle.
$$
Any unexplained notation and terminologies can be found in \cite{Ki} or \cite{OM}.

\section{Representations by octagonal forms}
\begin{lem}[Jones]\label{lemJones}
Let $\nu$ be a positive integer which is a multiple of 3.
If the Diophantine equation 
$$
x^2+2y^2=\nu
$$
has an integer solution $(x,y)$, then it also has an integer solution $(x_0,y_0)$ satisfying $x_0y_0\not\equiv 0\Mod 3.$
\end{lem}

\begin{proof}
See \cite{J}.
\end{proof}

Throughout this section, we use the following notations.
For $r=1,2,3,\dots$, define $H_r$ to be the set of all nonnegative integers less than $r$, i.e.,
$$
H_r=\{0,1,2,\dots,r-1\}.
$$
The map which assigns an integer $x$ to the remainder of $x$ divided by $r$ will be denoted by $\eta_r$.
Obviously, $\eta_r(x)$ is the smallest nonnegative integer congruent to $x$ modulo $r$.
For $g,h\in \n$, the map which assigns a matrix $(a_{ij})\in \mathcal{M}_{g,h}(\z)$ to the matrix $(\eta_r(a_{ij}))\in \mathcal{M}_{g,h}(H_r)$ will be denoted by $\eta_r$ also.
For a matrix $T=(t_{ij})\in \mathcal{M}_{g,h}(\z)$ and an integer $l$ with $1\le l\le g$, we denote by $[T]_l$ the $l$-th row vector $(t_{l1},t_{l2},\dots,t_{lh})$ of $T$.

The following method on the representation of an arithmetic progression by ternary quadratic forms have been developed through \cite{O1} and \cite{O2}.
We refer the reader \cite[Theorem 2.3]{JOS} who is interested in the most recent form of the method.
For a ternary $\z$-lattice $L$, we identify $L$ with its Gram matrix $M_L$ and we will describe the method in the level of matrices.
Let $M,N\in \mathcal{M}_3(\z)$ be symmetric matrices, $d$ a positive integer, and $a\in H_d$.
We define a set $S_{d,a}$ by
$$
S_{d,a}=\{ du+a : u\in \z_{\ge 0}\},
$$
and define
\begin{align*}
R(M,N,d)&=\left\{ T\in \mathcal{M}_3(\z) : {}^tTMT=d^2N \right\},\\
R(N,d,a)&=\left\{ v\in H_d^3 : \eta_r({}^tvNv)=a\right\},\\
R_M(N,d,a)&=\left\{ v\in R(N,d,a) : \eta_r(Tv)=\mathbf{0}\in H_d^3\ \ \text{for some}\ \ T\in R(M,N,d)\right\}.
\end{align*}
If $R(N,d,a)-R_M(N,d,a)=\emptyset$, then we write
$$
N\prec_{d,a}M.
$$

\begin{lem} \label{lemnobad}
Under the notations given above, if $N\prec_{d,a}M$, then
$$
S_{d,a}\cap Q(N)\subset Q(M).
$$
\end{lem}

\begin{proof}
See \cite[Theorem 2.3]{O1}.
\end{proof}

The following lemma is a matrix version of \cite[Corollary 2.2]{O2}.

\begin{lem} \label{lembad}
Under the notations given above, assume that there is a partition $R(N,d,a)-R_M(N,d,a)=P_1\cup P_2\cup \cdots \cup P_g$ satisfying the following properties; for each $i=1,2,\dots,g$, there is a matrix $T_i\in R(N,N,d)$ such that
\begin{enumerate} [(i)]
\item the matrix $\dfrac{1}{d}T_i$ is of infinite order;
\item $\left\{ \eta_d\left( \dfrac{1}{d}T_iv+x\cdot {}^t[T_i]_1+y\cdot {}^t[T_i]_2+z\cdot {}^t[T_i]_3\right) \right\} \subseteq P_i\cup R_M(N,d,a)$ for every $v\in P_j$ and every ${}^t(x,y,z)\in H_d^3$.
\end{enumerate}
Then we have
$$
S_{d,a}\cup Q(N)-\bigcup_{j=1}^g \left\{ {}^tw_iNw_i\cdot h^2 : h\in \n \right\} \subset Q(M),
$$
where $w_i={}^t(w_{i1},w_{i2},w_{i3})$ is a vector in $\z^3$ with $\gcd(w_{i1},w_{i2},w_{i3})=1$ such that $\dfrac{1}{d}T_iw_i=\det \left( \dfrac{1}{d}T_i\right) w_i$ for $i=1,2,\dots,g$.
\end{lem}

In Proposition \ref{propmany}, we will find the set of positive integers which are not represented by each of 26 octagonal forms in Table \ref{tablepm}.
Since the proofs for some octagonal forms among 26 forms are similar to each other, we divide 26 forms into eight batches and only provide proofs for eight representative cases.
Some lemmas which is needed only for proofs of the other 18 octagonal forms are also omitted and are available upon request to the authors.

\begin{lem} \label{lem111}
Let $\nu$ be a positive integer which is a multiple of 3 and is not of the form $2^{2a}(8b+7)$ for nonnegative integers $a$ and $b$.
Then we have $\nu\rat \langle 1,1,1\rangle$.
\end{lem}

\begin{proof}
Since $\nu$ is not of the form $2^{2a}(8b+7)$ for some $a,b\in \z_{\ge 0}$, there is a vector $(x_0,y_0,z_0)\in \z^3$ such that
$$
\nu=x_0^2+y_0^2+z_0^2
$$
by Legendre's three-square theorem.
Note that
\begin{equation} \label{eqlem111}
x_0y_0z_0\not\equiv 0\Mod 3 \ \ \text{or} \ \ x_0\equiv y_0\equiv z_0\equiv 0\Mod 3
\end{equation}
since $\nu$ is a multiple of 3.
If $\nu \not\equiv 0\Mod 9$, then $x_0y_0z_0\not\equiv 0\Mod 3$ and we are done.
So we may write $\nu=9\nu'$ with $\nu'\in \n$.
Using \eqref{eqlem111}, one may easily show that
\begin{align*}
&\left\vert \left\{ (x,y,z)\in \z^3 : x^2+y^2+z^2=9\nu',\ xyz\not\equiv 0\Mod 3 \right\} \right\vert \\
&=r(9\nu',\langle 1,1,1\rangle)-r(\nu',\langle 1,1,1\rangle).
\end{align*}
Note that the class number of the $\z$-lattice $\langle 1,1,1\rangle$ is one and thus we have 
$$
r(9\nu',\langle 1,1,1\rangle)-r(\nu',\langle 1,1,1\rangle)>0
$$
by \cite[Lemma 2.2]{KyO}.
Therefore, we have
\begin{align*}
&\left\vert \left\{ (x,y,z)\in \z^3 : x^2+y^2+z^2=9\nu',\ xyz\not\equiv 0\Mod 3 \right\} \right\vert \\
&=r(9\nu',\langle 1,1,1\rangle)-r(\nu',\langle 1,1,1\rangle) \\
&>0.
\end{align*}
This completes the proof.
\end{proof}

\begin{lem} \label{lem123}
Let $\nu$ be a positive integer which is a multiple of 6 and not of the form $2^{2a+1}(8b+5)$ for nonnegative integers $a$ and $b$.
Then we have $\nu\rat \langle 1,2,3\rangle$.
\end{lem}

\begin{proof}
Note that the class number of $\langle 1,2,3\rangle$ is one.
From this and the assumption that $\nu$ is not of the form $2^{2a+1}(8b+5)$, one may easily check that $\nu$ is represented by $\langle 1,2,3\rangle$.
Thus there is a vector $(x,y,z)\in \z^3$ such that
$$
\nu=x^2+2y^2+3z^2.
$$
We may assume that $x^2+2y^2>0$ because of the equation $3z^2=z^2+2z^2$.
Since $\nu-3z^2\equiv 0\Mod 3$, there is a vector $(x_1,y_1)$ with $x_1y_1\not\equiv 0\Mod 3$ such that
$$
x^2+2y^2=x_1^2+2y_1^2
$$
by Lemma \ref{lemJones}.
We assume that $z\equiv 0\Mod 3$ since otherwise we are done.
Note that $x_1\equiv z\Mod 2$ since $\nu=x_1^2+2y_1^2+3z^2$ is even.
If we define
$$
x_2=\frac{x_1+3z}{2}\ \ \text{and}\ \ z_2=\frac{x_1-z}{2},
$$
then one may easily show that $\nu=x_2^2+2y_1^2+3z_2^2$ and $x_2y_1z_2\not\equiv 0\Mod 3$.
This completes the proof.
\end{proof}

\begin{lem} \label{lem113}
Let $\nu$ be a positive integer congruent to 5 or 8 modulo 12.
Then we have $\nu\rat \langle 1,1,3\rangle$.
\end{lem}

\begin{proof}
Let $L(1)=\langle 1\rangle \perp \begin{pmatrix} 4&1\\1&7\end{pmatrix}$ be a ternary $\z$-lattice.
Since $\nu\equiv 2\Mod 3$, one may easily check that $\nu\ra \gen(L(1))$.
Note that the class number of $L(1)$ is three and
$$
\gen(L(1))=\left\{ L(1),\ L(2)=\begin{pmatrix}2&-1&1\\-1&4&1\\1&1&5\end{pmatrix},\ L(3)=\langle 1,1,27\rangle \right\}.
$$
One may show that $N\prec_{d,a}M$ for
$$
M=L(1),\ N\in \{L(2),L(3)\},\ d=4,\ a\in \{0,1\}.
$$
Since $\nu\equiv 0,1\Mod 4$ by assumption, we have $\nu\ra L(1)$ by Lemma \ref{lemnobad}.
Thus there is a vector $(x,y,z)\in \z^3$ such that
\begin{align*}
\nu&=x^2+4y^2+7z^2+2yz\\
&=x^2+(y-2z)^2+3(y+z)^2.
\end{align*}
Since $\nu\equiv 2\Mod3$, we have $x(y-2z)\not\equiv 0\Mod 3$ and thus $y+z\not\equiv 0\Mod 3$.
This completes the proof.
\end{proof}

\begin{lem} \label{lem233}
Let $\nu$ be a positive integer congruent to 8 or 14 modulo 24.
Then $\nu\rat \langle 2,3,3\rangle$.
\end{lem}

\begin{proof}
The proof is quite similar to that of Lemma \ref{lem113}.
One may easily prove the lemma by using
$$
L(1)=\begin{pmatrix}8&2&1\\2&14&7\\1&7&17\end{pmatrix},\ L(2)=\begin{pmatrix}5&1&0\\1&11&0\\0&0&27\end{pmatrix},\ L(3)=\langle 2,27,27\rangle,
$$
$d=24$, $a\in \{8,14\}$, and an equation
$$
8x^2+14y^2+17z^2+4xy+2xz+14yz=2(x-2y-z)^2+3(x+y+2z)^2+3(x+y-z)^2.
$$
\end{proof}

\begin{lem} \label{lem346}
Let $\nu$ be a positive integer greater than or equal to 13, congruent to 16 or 22 modulo 24, not of the form $2^{2a+1}(8b+7)$ for nonnegative integers $a$ and $b$, and is not a perfect square.
Then we have $\nu\rat \langle 3,4,6\rangle$.
\end{lem}

\begin{proof}
The proof is quite similar to that of Lemma \ref{lem113}.
One may easily prove the lemma by using
$$
L(1)=\langle 4,9,18\rangle,\ L(2)=\langle 1,18,36\rangle,\ L(3)=\begin{pmatrix}7&3&1\\3&9&3\\1&3&13\end{pmatrix},
$$
$d=24$, $a\in \{16,22\}$, and an equation
$$
4x^2+9y^2+18z^2=4x^2+3(y+2z)^2+6(y-z)^2
$$
with Lemma \ref{lemJones}.
\end{proof}

\begin{lem} \label{lem234}
Let $\nu$ be a positive integer which is a multiple of 3 satisfying at least one of the followings;
\begin{enumerate} [(i)]
\item $\nu\equiv 2\Mod {16}$ or $\nu\equiv 8,24,56\Mod {64}$;
\item $\nu\equiv 3\Mod {9}$ and $\nu\not\in \{ 2^{2a+1}(8b+5) : a,b\in \z_{\ge 0} \} \cup \{ 3c^2 : c\in \n \}$.
\end{enumerate}
Then we have $\nu\rat \langle 2,3,4\rangle$.
\end{lem}

\begin{proof}
Let $L=\langle 2,3,4\rangle$ and assume first that $(i)$ holds.
Note that the class number of $L$ is two and
$$
\gen(L)=\{ L,\ \langle 1,2,12\rangle \}.
$$
By assumption, one may easily check that $\nu\ra \gen(L)$.
Since $\nu$ is even, we have $\nu\ra \langle 2,4,12\rangle$.
If we put $\nu=2\nu'$, then $\nu'\ra \langle 1,2,6\rangle$ and thus there is a vector $(x,y,z)\in \z^3$ such that
$$
\nu'=x^2+2y^2+6z^2.
$$
We may assume that $x^2+2y^2>0$ since $6z^2=(2z)^2+2z^2$.
From Lemma \ref{lemJones} follows that there is a vector $(x_1,y_1)\in \z^2$ with $x_1y_1\not\equiv 0\Mod 3$ such that
$$
x^2+2y^2=x_1^2+2y_1^2.
$$
Now we may write
$$
\nu'=x_1^2+2y_1^2+6z^2.
$$
Since $\nu'\equiv 1,4\Mod 8$, we have $y_1\equiv z\Mod 2$.
We may assume that $z \equiv 0\Mod 3$ since otherwise we are done.
Then we have
$$
\nu'=x_1^2+2\left( \frac{y_1+3z}{2}\right)^2+6\left( \frac{y_1-z}{2} \right)^2,
$$
where $x_1\cdot \dfrac{y_1+3z}{2}\cdot \dfrac {y_1-z}{2} \not\equiv 0\Mod 3$.
It follows that
$$
\nu=2x_1^2+3(y_1-z)^2+4\left( \frac{y_1+3z}{2} \right)^2
$$
and thus we have $\nu\rat L$.

Next, we assume that (ii) holds and define a ternary $\z$-lattice $L(1)$ by
$$
L(1)=\begin{pmatrix}9&3&0\\3&15&6\\0&6&18\end{pmatrix} 
$$
Note that the class number of $K$ is three and
$$
\gen(L(1))=\spn(L(1))=\{L(1),\ L(2)=\langle 6,12,27\rangle,\ L(3)=\langle 3,6,108\rangle \}.
$$
We claim that $\nu \ra L(1)$.
One may easily show that the claim implies that $\nu \rat L$ by using the equation
\begin{align*}
&9x_2^2+15y_2^2+18z_2^2+6x_2y_2+12y_2z_2\\
&=2(x_2-2y_2-z_2)^2+3(x_2+y_2+2z_2)^2+4(x_2+y_2-z_2)^2
\end{align*}
and the assumption that $\nu \equiv 3\Mod 9$.

One may easily check from the assumption (ii) that $\nu\ra \gen(L(1))$.
To prove the claim, assume first that $\nu \ra L(2)$.
We use Lemma \ref{lembad} with
$$
M=L(1),\ N=L(2),\ d=9,\ a=3.
$$
Note that
$$
R(N,d,a)-R_M(N,d,a)=\left\{ {}^t(0,v_2,v_3)\in \Psi_9^3 : v_2\not\equiv 0\Mod 3,\ v_3\equiv 0\Mod 3 \right\}.
$$
If we take $P_1=R(N,d,a)-R_M(N,d,a)$ and
$$
T_1=\frac{1}{9}\begin{pmatrix}3&0&-18\\0&9&0\\4&0&3\end{pmatrix},
$$
then one may easily check that all conditions of Lemma \ref{lembad} are satisfied.
Hence $\nu \ra L(1)$ unless $\nu=12h^2$ for some $h\in \n$.
However, we have $12h^2\ra L(1)$ since $12\ra K$.
So far, we proved that $\nu \ra L(1)$ if $\nu \ra L(2)$ is the case.
Now we assume that $\nu \ra L(3)$.
We use Lemma \ref{lembad} with
$$
M=L(1),\ N=L(3),\ d=9,\ a=3.
$$
Note that
$$
R(N,d,a)-R_M(N,d,a)=\left\{ {}^t(v_1,0,v_3)\in \Psi_9^3 : v_1\not\equiv 0\Mod 3,\ v_3\equiv 0\Mod 3 \right\}.
$$
If we take $P_1=R(N,d,a)-R_M(N,d,a)$ and
$$
T_1=\frac{1}{9}\begin{pmatrix}9&0&0\\0&3&-36\\0&2&3\end{pmatrix},
$$
then one may easily check that all conditions of Lemma \ref{lembad} are satisfied.
Hence $\nu \ra L(1)$ unless $\nu=3h^2$ for some $h\in \n$.
If $h\equiv 0\Mod 2$, then $\nu \ra L(1)$ since $12\ra L(1)$.
If $h\equiv 1\Mod 2$, then since $\nu \ge 12$ and $\nu \equiv 3\Mod 9$ by assumptions,
there is a prime $p$ greater than 3 such that $h\equiv 0\Mod p$.
Since $3\ra L(3)$, we have $3p^2\ra L(1)$ by \cite[Theorem 1]{BH}.
From this follows that $\nu \ra L(1)$.
Hence we have the claim and now the lemma follows.
\end{proof}
\begin{lem} \label{lem334}
Let $\nu$ be a positive integer greater than or equal to 10, congruent to 7 modulo 24.
Then we have $\nu\rat \langle 3,3,4\rangle$.
\end{lem}

\begin{proof}
The proof is similar to that of the second case of Lemma \ref{lem234}.
One may use 
$$
L(1)=\begin{pmatrix}10&-1&1\\-1&19&8\\1&8&19\end{pmatrix},\ L(2)=\begin{pmatrix}4&0&0\\0&27&0\\0&0&27\end{pmatrix},\ L(3)=\begin{pmatrix}7&2&0\\2&16&0\\0&0&27\end{pmatrix}.
$$
For the case when $\nu \ra L(2)$, we note that $L(2)\prec_{8,7}L(1)$.
In the case when $\nu \ra L(3)$, one may use Lemma \ref{lembad} with
$$
M=L(1),\ N=L(3),\ d=8,\ a=7
$$
to show that $\nu \ra L(1)$.
\end{proof}

\begin{lem} \label{lem356}
Let $\nu$ be a positive integer greater than or equal to 14, congruent to 8,32,38 modulo 48, and not divisible by 5.
Then we have $\nu\rat \langle 3,5,6\rangle$.
\end{lem}

\begin{proof}
Let $L(1)=\langle 5,9,18\rangle$.
One may easily show that $\nu \ra L(1)$ is a sufficient condition for $\nu \rat \langle 3,5,6\rangle$, by using the equation
$$
5x^2+9y^2+18z^2=5x^2+3(y+2z)^2+6(y-z)^2
$$
and Lemma \ref{lemJones}.
To prove that $\nu \ra L(1)$, one may use
$$
\gen(L(1))=\spn(L(1))=\left\{L(1),\ L(2)=\langle 2,9,45\rangle,\ L(3)=\begin{pmatrix}8&3&2\\3&9&3\\2&3&14\end{pmatrix}\right\}.
$$
For the case when $\nu \equiv 38\Mod{48}$, we note that
$$
L(2)\prec_{48,38}L(1),\ \ L(3)\prec_{48,38}L(1).
$$
When $\nu \equiv 8,32\Mod{48}$, using Lemma \ref{lembad} with
$$
M=L(1),\ N\in \{L(2),L(3)\},\ d=24,\ a=8,
$$
one may show that $\nu \ra L(1)$ by a similar argument in the proof of the second case of Lemma \ref{lem234}.
\end{proof}

We will use the following lemma which is a slight generalization of the method appeared in the proof of \cite[Theorem 2.1]{JO}.
Though the proof is essentially the same as in that paper, we provide the proof for completeness.
\begin{lem} \label{lem11}
Let $a, b$ be positive integers with $a\equiv b\not\equiv 0\Mod 3$ and let $l,\alpha, \beta$ be integers defined by $l=\text{lcm}(a,b)$, $\alpha=\dfrac{(a+b)l^2}{ab}$ and $\beta=\dfrac{\alpha-a-b-6}{3}$.
Let $u$ be a positive integer and assume that there is an integer $w$ such that $u-\alpha P_8(w)-\beta$ is congruent to 2 modulo 3 and represented by the diagonal ternary quadratic form $\langle 1,1,3(a+b)\rangle$.
Then $u$ is represented by the quaternary octagonal form $p_8(3,3,a,b)$.
\end{lem}

\begin{proof}
It suffices to show that
$$
3u+3+3+a+b\rat \langle 3,3,a,b\rangle.
$$
By assumptions, there is a vector $(x,y,z)\in \z^3$ with $xy\not\equiv 0\Mod 3$ such that
$$
u-\alpha(3w^2-2w)-\frac{\alpha-a-b-6}{3}=x^2+y^2+3(a+b)z^2.
$$
Multiplying both sides of this equation by 3, we get
$$
3u-(3w-1)^2\alpha+3+3+a+b=3x^2+3y^2+9(a+b)z^2.
$$
It follows that
\begin{align*}
3u+3+3+a+b&=3x^2+3y^2+9(a+b)z^2+(3z-1)^2(a+b)\frac{l^2}{ab}\\
&=3x^2+3y^2+a\left(3z+(3w-1)\frac{l}{a}\right)^2+b\left(3z-(3w-1)\frac{l}{b} \right)^2.
\end{align*}
Note that $l\not\equiv 0\Mod 3$ since $ab\not\equiv 0\Mod 3$.
This completes the proof.
\end{proof}

\begin{lem} \label{lem2233}
The octagonal form $p_8(2,2,3,3)$ represents all positive integers $u$ satisfying $u\not\equiv 1\Mod 4$ and $u\neq 11,14$.
\end{lem}

\begin{proof}
Let $u$ be a positive integer greater than 17 with $u\not\equiv 1\Mod 4$.
Define an integer $w$ by
$$
w=\begin{cases}0&\text{if}\ \ u\equiv 0\Mod 3,\\
1&\text{if}\ \ u\equiv 1\Mod 3,\\
-1&\text{if}\ \ u\equiv 2\Mod 3.\end{cases}
$$
Note that the class number of the ternary $\z$-lattice $L=\langle 1,1,12\rangle$ is one and thus one may easily check that every positive integer congruent to 2 modulo 3 and not congruent to 3 modulo 4 is represented by $L$.
Now one may easily check that the quadruple $(a,b,u,w)$ satisfies all conditions of Lemma \ref{lem11}.
Thus $u$ is represented by $p_8(2,2,3,3)$.
On the other hand, one may directly check that $p_8(2,2,3,3)$ represents every positive integer in the set
$$
\{2,3,4,6,7,8,10,12,15,16\}.
$$
This completes the proof.
\end{proof}

\begin{lem} \label{lem2233t}
Let $t$ be a positive integer not divisible by 4. The octagonal form $p_8(2,2,3,3,t)$ represents all positive integers greater than or equal to $t+15$.
\end{lem}

\begin{proof}
Let $u$ be a positive integer greater than or equal to $t+15$.
If $u\not\equiv 1\Mod 4$, then $u\ra p_8(2,2,3,3)$ by Lemma \ref{lem2233}, and thus
$$
u\ra p_8(2,2,3,3,t).
$$
If $u\equiv 1\Mod 4$, then $u-t$ is an integer greater than or equal to 15 not congruent to 1 modulo 4.
By Lemma \ref{lem2233} again, we have $u-t\ra p_8(2,2,3,3)$.
Thus $u\ra p_8(2,2,3,3,t)$.
This completes the proof.
\end{proof}

\noindent Method {\bf (i)} Let $p_8(a_1,a_2,\dots,a_k)$ be a $k$-ary octagonal form with $k\ge 4$.
Let $L=\langle a_1,a_2,\dots,a_k\rangle$ and put $\alpha=a_1+a_2+\cdots+a_k$.
We take a vector $(b_1,b_2,\dots,b_l)$ of length $l\le k-3$ with
$$
(b_1,b_2,\dots,b_l)\prec (a_1,a_2,\dots,a_k).
$$
After rearrangement, we may assume that $b_i=a_i$ for $i=1,2,\dots,l$.
Define $K=\langle c_1,c_2,c_3\rangle$ as following:
When $l=k-3$, then
$$
\langle c_1,c_2,c_3\rangle \simeq \left\langle \frac{a_{l+1}}{d},\frac{a_{l+2}}{d},\frac{a_{l+3}}{d}\right\rangle,
$$ 
where $d=\gcd(a_{l+1},a_{l+2},a_{l+3})$.
When $l\le k-4$, 
$$
\langle c_1,c_2,c_3\rangle \simeq \left\langle \frac{\sum_{i_1\in I_1}a_{i_1}}{d},\frac{\sum_{i_2\in I_2}a_{i_2}}{d},\frac{\sum_{i_3\in I_3}a_{i_3}}{d} \right\rangle,
$$
where $d=\gcd \left( \sum_{i_1\in I_1}a_{i_1},\sum_{i_2\in I_2}a_{i_2},\sum_{i_3\in I_3}a_{i_3}\right)$ and
$$
\{l+1,l+2,\dots,k\}=I_1\cup I_2\cup I_3
$$
is a partition.
Note that the conditions in any of Lemmas \ref{lem111}-\ref{lem356} consist of some congruence relations with a lower bound.
We find a set $B(i)$ of positive integers coprime to 3 for each $i=1,2,\dots,l$, and find a set of positive integers $G=\{g_1,g_2,\dots,g_r\}$ with $g_s\in \{ dt^2 : t\in \n-3\n \}$ satisfying the following; for any sufficiently large positive integer $u$ $(u\ge u_0)$, there is a vector $(x_1,x_2,\dots,x_l)\in \prod_{1\le i\le l} B(i)$ and a positive integer $g_j\in G$ such that
$$
\nu'=\frac{3u+\alpha-b_1x_1^2-b_2x_2^2-\cdots-b_lx_l^2}{g_j}
$$
is a positive integer satisfying all conditions of the lemma regarding representations by the ternary quadratic form $K$.
Now one get $\nu' \rat K$ and from this follows that $\nu \rat L$, or equivalently $u\ra p_8(a_1,a_2,\dots,a_k)$.
For all integers $v$ less than $u_0$, with the help of computer, we directly check whether it is represented by $p_8(a_1,a_2,\dots,a_k)$ or not.
\vskip 10pt
\noindent Method {\bf (ii)} Let $p_8(a_1,a_2,\dots,a_k)$ be a $k$-ary octagonal form with $k\ge 5$ such that $(3,3,a,b)\prec (a_1,a_2,\dots,a_k)$ for some positive integers $a$ and $b$ satisfying $a\equiv b\not\equiv 0\Mod 3$.
After rearrangement, we may assume that $(a_1,a_2,a_3,a_4)=(3,3,a,b)$.
For each sufficiently large integer $\nu$, we find integers $b_5,b_6,\dots,b_k$ and $w$ such that
$$
\nu-a_5P_8(b_5)-a_6P_8(b_6)-\cdots-a_kP_8(b_k)-\alpha P_8(w)-\beta
$$
is congruent to 2 modulo 3 and represented by $\langle 1,1,3(a+b)\rangle$, where $\alpha$ and $\beta$ is defined as in Lemma \ref{lem11}.
Then we apply Lemma \ref{lem11}.

\begin{lem} \label{lem3456}
Let $\nu$ be a positive integer divisible by 3 satisfying at least one of the followings;
\begin{enumerate} [(i)]
\item $\nu>99$ and $\nu\not\equiv 1,2,7,9,14,15\Mod{16}$;
\item $\nu>24$ and $\nu\equiv 6\Mod{18}$;
\item $\nu\equiv 0\Mod{18}$ and $\nu\not\equiv 14\Mod{16}$. 
\end{enumerate}
Then $\nu\rat \langle 3,4,5,6\rangle$.
\end{lem}

\begin{proof}
First, assume that (i) holds.
If $\nu\le 615$, then one may directly check that
$$
\nu\rat \langle 3,4,5,6\rangle.
$$
We may assume further that $\nu\ge 618$.
Suppose that $\nu\equiv 3,5\Mod 8$.
One may easily check that there is an integer $x_1\in \{1,5,7,11\}$ such that $\nu-5x_1^2$ is an integer greater than or equal to 13 satisfying
$$
\nu-5x_1^2\equiv 6,8,22,24,38,40\ \ \text{or}\ \ 54\Mod{64}.
$$
Then by Lemma \ref{lem346}, we have $\nu-5x_1^2\rat \langle 3,4,6\rangle$.
It follows that $\nu\rat \langle 3,4,5,6\rangle$.
Now suppose that $\nu\equiv 6,10\Mod{16}$ or $\nu\equiv 0\Mod 4$.
One may easily check that there is an integer $x_2\in \{1,2,4,5\}$ such that $\nu-4x_2^2$ is an integer greater than or equal to 14 satisfying
$$
\nu-4x_2^2\not\equiv 0\Mod 5\ \ \text{and}\ \ \nu-4x_2^2\equiv 0,6,8\Mod{16}.
$$
Then by Lemma \ref{lem356}, we have $\nu-4x_2^2\rat \langle 3,5,6\rangle$.
It follows that $\nu\rat \langle 3,4,5,6\rangle$.

Second, assume that (ii) holds.
Note that there is no integer $u$ less than or equal to 99 satisyfing
$$
u\equiv 6\Mod{18},\ u\in \{2^{2a+2}(8b+5) : a,b\in \z_{\ge 0} \}.
$$
If $u$ is an integer greater than 99 satisfying
$$
u\equiv 6\Mod{18},\ u\in \{2^{2a+2}(8b+5) : a,b\in \z_{\ge 0} \},
$$
then $u\rat \langle 3,4,5,6\rangle$ by the case (i).
One may directly check that $6c^2\rat \langle 3,4,5,6\rangle$ for $c=3,4$.
Note that $6c^2\equiv 0,6,8\Mod{16}$, and thus we have $6c^2\rat \langle 3,4,5,6\rangle$ for every integer $c\ge 5$ by the case (i).
Thus we may assume that
$$
\nu\equiv 6\Mod{18},\ \nu\not\in \{ 2^{2a+2}(8b+5) : a,b\in \z_{\ge 0}\} \cup \{6c^2 : c\in \n \}.
$$
It follows immediately from Lemma \ref{lem234} that
$$
\frac{\nu}{2}\rat \langle 2,3,4\rangle.
$$
It follows that $\nu\rat \langle 4,6,3+5\rangle$, and thus we have $\nu\rat \langle 3,4,5,6\rangle$.

Last, assume that (iii) holds.
One may directly check that
$$
\nu\rat \langle 3,4,5,6\rangle
$$
for every $\nu\in \{18,36,54,72,90\}$.
So we may assume that $\nu\ge 108$.
Note that $\nu\rat \langle 3,4,5,6\rangle$ if $\nu\equiv 0\Mod 8$ by the case (i).
Thus we may assume that
$$
\nu\equiv 0\Mod{18}\ \text{and}\ \nu\not\in \{ 2^{2a+1}(8b+7) : a,b\in \z_{\ge 0}\}.
$$
If we put $\nu=3\nu'$, then $\nu'$ is an integer greater than 33 such that
$$
\nu' \equiv 0\Mod 6\ \text{and}\ \nu' \not\in \{ 2^{2a+1}(8b+5) : a,b\in \z_{\ge 0}\}.
$$
It follows immediately from this and Lemma \ref{lem123} that $\nu' \rat \langle 1,2,3\rangle$.
It follows that $\nu\rat \langle 3,6,4+5\rangle$, and thus we have $\nu\rat \langle 3,4,5,6\rangle$.
This completes the proof.
\end{proof}

\begin{table}[ht]
\caption{The set $Z$ of integers in $\mathcal T(a_1)$ not represented by $p_8(a_1,a_2,\dots,a_k)$}
\label{tablepm}
\begin{center}
\begin{tabular}{|ccccccccc|c|c|}
\Xhline{1pt}
$a_1$&$a_2$&$a_3$&$a_4$&$a_5$&$a_6$&$a_7$&$a_8$&$a_9$&$Z$&Method\\
\hline
2&2&2&3&&&&&&$\{8,11\}$&(i)\\
\hline
2&2&3&4&&&&&&$\emptyset$&(i)\\
\hline
2&2&3&6&&&&&&$\{14\}$&\\
\hline
2&3&3&4&&&&&&$\{11\}$&(i)\\
\hline
2&3&4&4&&&&&&$\{12\}$&(i)\\
\hline
2&3&4&5&&&&&&$\emptyset$&(i)\\
\hline
2&3&4&6&&&&&&$\{18\}$&\\
\hline
2&3&4&8&&&&&&$\emptyset$&(i)\\
\hline
2&2&3&3&3&&&&&$\{14\}$&\\
\hline
3&3&4&4&5&&&&&$\{17,21\}$&(ii)\\
\hline
3&3&4&5&6&&&&&$\emptyset$&(i)\\
\hline
3&3&4&5&10&&&&&$\emptyset$&(ii)\\
\hline
3&4&4&5&6&&&&&$\emptyset$&(i)\\
\hline
3&4&5&6&6&&&&&$\{22\}$&(i)\\
\hline
3&4&5&6&8&&&&&$\emptyset$&(i)\\
\hline
3&4&5&6&9&&&&&$\{36\}$&(i)\\
\hline
3&4&5&6&10&&&&&$\{27\}$&\\
\hline
3&4&5&6&12&&&&&$\emptyset$&(i)\\
\hline
4&4&5&6&7&&&&&$\{23,28\}$&(i)\\
\hline
4&5&6&7&8&&&&&$\emptyset$&(i)\\
\hline
5&5&6&7&8&9&&&&$\emptyset$&(i)\\
\hline
5&6&7&8&9&10&&&&$\emptyset$&(i)\\
\hline
6&6&7&8&9&10&11&&&$\emptyset$&(i)\\
\hline
6&7&8&9&10&11&12&&&$\emptyset$&(i)\\
\hline
7&8&9&10&11&12&13&14&&$\emptyset$&(i)\\
\hline
8&9&10&11&12&13&14&15&16&$\emptyset$&(i)\\
\Xhline{1pt}
\end{tabular}
\end{center}
\end{table}

\begin{prop} \label{propmany}
In Table \ref{tablepm}, for each octagonal form $g=p_8(a_1,a_2,\dots,a_k)$ with $a_1\le a_2\le \cdots \le a_k$, the set of positive integers greater than or equal to $a_1$ which is not represented by $g$ is equal to $Z$.
\end{prop}

\begin{proof}
{\bf (Case 1)} The first case is a typical case of using Method (i) and we choose $p_8(2,2,2,3)$ as the representative.
Put $\alpha=2+2+2+3=9$ and let $L=\langle 2,2,2,3\rangle$.
It suffices to show that $3u+\alpha \rat L$ for every positive integer $u\ge 2$ except when $u\in Z=\{8,11\}$.
Let $v$ be a positive integer greater than $u_0=166$ and put $\nu=3v+\alpha$.
Then $\nu$ is a positive integer greater than 507 such that $\nu\equiv 0\Mod 3$.
Put $b_1=3$, $d_1=2$ and
$$
H=S_{8,4} \cup S_{16,2} \cup S_{16,6} \cup S_{16,10} \cup S_{32,16}.
$$
One may easily check that there is an integer $x$ with
$$
x\in B(1)=\{1,2,4,5,7,13\}
$$
such that $\nu-b_1x^2\in H$.
Then
$$
\frac{\nu-b_1x^2}{d_1}\in \frac{1}{d_1}H=\left\{ \frac{v}{d_1} : v\in H \right\}=S_{4,2} \cup S_{8,1} \cup S_{8,3} \cup S_{8,5} \cup S_{16,8}.
$$
By Lemma \ref{lem111}, we have 
$$
\frac{\nu-b_1x^2}{d_1}\rat K=\langle 1,1,1\rangle.
$$
It follows that
$$
\nu-b_1x^2\rat \langle 2,2,2\rangle
$$
and thus we have
$$
\nu\rat L=\langle 2,2,2,3\rangle
$$
since $x\not\equiv 0\Mod 3$.
One may directly check that
$$
3u+\alpha \rat L
$$
for every positive integer $u$ with $2\le u\le u_0=166$ except for $u\in Z=\{8,11\}$.
This completes the proof for the case $g=p_8(2,2,2,3)$.

\noindent {\bf (Case 2)} In some of the cases, one may take more than one vectors $(b_1,b_2,\dots,b_l)$ depending on the congruence class of $\nu$ in Method (i).
We take $p_8(2,3,3,4)$ as the representative of this case.
Let $v$ be a positive integer greater than 117 and put $\nu=3v+12$.
Then $\nu$ is a positive integer greater than 363 such that $\nu\equiv 0\Mod 3$.

First, assume that $\nu\equiv 0\Mod 2$.
One may easily check that there is an integer $x\in \{1,2\}$ such that
such that $\nu-4x^2$ is a positive integer congruent to 0 or 6 modulo 8.
By Lemma \ref{lem233}, we have $\nu-4x^2\rat \langle 2,3,3\rangle$.
Thus $\nu\rat \langle 2,3,3,4\rangle$.

Second, assume that $\nu\equiv 1,7\Mod 8$.
One may easily check that there is an integer $x\in \{1,2\}$ such that
such that $\nu-2x^2$ is a positive integer congruent to 7 modulo 8.
By Lemma \ref{lem334}, we have $\nu-2x^2\rat \langle 3,3,4\rangle$.
Thus $\nu\rat \langle 2,3,3,4\rangle$.

Last, assume that $\nu\equiv 3,5\Mod 8$.
One may easily check that there is an integer $x\in \{1,5,7,11\}$ such that
such that $\nu-3x^2$ is a positive integer satisfying
$$
\nu-3x^2\equiv 2,8,18,24,34,50,56\Mod{64}.
$$
By Lemma \ref{lem234}, we have $\nu-3x^2\rat \langle 2,3,4\rangle$.
Thus $\nu\rat \langle 2,3,3,4\rangle$.

One may directly check that
$$
3u+12\rat \langle 2,3,3,4\rangle
$$
for every positive integer $u$ with $2\le u\le 117$ except $u=12$.

\noindent {\bf (Case 3)} The third case is to use Method (i) with a technique which is useful when we use Lemma \ref{lem234} with assumption (ii) in that lemma.
We pick $p_8(2,3,4,5)$ as the representative.
Let $v$ be a positive integer greater than 14743 and put $\nu=3v+14$.
Then $\nu$ is a positive integer greater than 44243 such that $\nu\equiv 2\Mod 3$.
One may easily check that there is a positive integer $x$ with
$$
x\in \{1,2,4,5,7,8\}
$$
such that $\nu-5x^2\equiv 3\Mod{18}$.
It is clear that
$$
\nu-5(18-x)^2\equiv \nu-5x^2\equiv 3\Mod{18}.
$$
We assert that at least one of $\nu-5x^2$ and $\nu-5(18-x)^2$ is not of the form $3c^2$ for some $c\in \n$.
To prove the assertion, suppose that
$$
\nu-5x^2=3a^2\ \ \text{and}\ \ \nu-5(18-x)^2=3b^2
$$
for some positive integers $a$ and $b$ with $b<a$.
Then $3(a^2-b^2)=1620-180x$ and it follows that
$$
(a-b)(a+b)=60(9-x)\in \{ 60,120,240,300,420,480\}.
$$
Thus $a-b$ must be an even positive integer.
Since $a-b\ge 2$, we have $a+b\le 240$.
From this and $b\le a-2$, it follows that $a\le 121$.
Then
$$
\nu\le 3a^2+5x^2\le 3\cdot 121^2+5\cdot 8^2=44243,
$$
which is absurd.
Thus we have the assertion.
Now by Lemma \ref{lem234}, we have
$$
\nu-5x^2\rat \langle 2,3,4\rangle.
$$
Therefore, $\nu\rat \langle 2,3,4,5\rangle$.
One may directly check that 
$$
3u+14\rat \langle 2,3,4,5\rangle
$$
for every positive integer $u$ with $2\le u\le 14743$.

\noindent {\bf (Case 4)} The fourth case is to use Method (ii).
We take $p_8(3,3,4,5,10)$ as the representative.
Let $u$ be an integer greater than 3620.
We first assert that the ternary quadratic form $L=\langle 1,1,42\rangle$ represents every positive integer $g$ satisfying
\begin{equation} \label{eq334510}
g\equiv 2\Mod 3,\ g\not\equiv 0\Mod 7,\ g\equiv 2,4,10,12,14\Mod{16}.
\end{equation}
and prove the proposition using this assertion.
To prove the assertion, let $g$ be a positive integer satisfying the above congruence conditions.
One may easily check that $g\ra \gen(L)$.
Note that the class number of $L$ is two and
$$
\gen(L)=\left\{ L,\ M=\langle 2\rangle \perp \begin{pmatrix}2&1\\1&11\end{pmatrix}\right\}.
$$
If $g\ra M$, there is a vector $(x,y,z)\in \z^3$ such that
$$
g=2x^2+2y^2+11z^2+2yz.
$$
Since $g$ is even, so is $z$.
It follows that
\begin{align*}
g&=2x^2+2y^2+44z_1^2+4yz_1\\
&=(x+y+z_1)^2+(x-y-z_1)^2+42z_1^2,
\end{align*}
where $z=2z_1$.
Thus $g\ra L$ and we have the assertion.
One may easily check that there is a vector $(v,w)$ with
$$
v\in \{-1,0,1,2\},\quad w\in \{-2,-1,0,1,2,3\}
$$
such that $u-5P_8(v)-140P_8(w)-40$ is a positive integer satisfying \eqref{eq334510}.
Then
$$
u-5P_8(v)\ra p_8(3,3,4,10)
$$
by Lemma \ref{lem11} with $a=4$ and $b=10$.
Thus $u\ra p_8(3,3,4,5,10)$.
One may directly check that every positive integer $u$ with $3\le u\le 3620$ is represented by $p_8(3,3,4,5,10)$.

\noindent {\bf (Case 5)} The fifth case is when $g=p_8(2,2,3,6)$.
Let $v$ be a positive integer greater than 863 and put $\nu=3v+13$.
Then $\nu$ is a positive integer greater than 2602 such that $\nu\equiv 1\Mod 3$.
First, assume that $\nu$ is odd.
One may easily check that there is an integer $z$ with
$$
z\in \{1,5,11,17\}
$$
such that $\nu-9z^2$ is a positive integer satisfying
$$
\nu-9z^2\equiv 4,6,10,12,20,22,24,26,28,36,38,40,42,44,52,54,58\ \ \text{or}\ \ 60\Mod{64}.
$$
Since the class number of $\langle 2,2,18\rangle$ is one, one may easily check that
$$
\nu-9z^2\ra \langle 2,2,18\rangle.
$$
So there is a vector $(x,y,w)\in \z^3$ such that
$$
\nu-9z^2=2x^2+2y^2+18w^2.
$$
Thus we have
\begin{align*}
\nu&=2x^2+2y^2+9z^2+18w^2\\
&=2x^2+2y^2+3(z+2w)^2+6(z-w)^2.
\end{align*}
Note that $xy\not\equiv 1\Mod 3$ since $\nu\equiv 1\Mod 3$.
If $w\equiv 0\Mod 3$, then $(z+2w)(z-w)\not\equiv 0\Mod 3$ and we are done.
Suppose that $w\not\equiv 0\Mod 3$.
Then by changing the sign of $w$ if necessary, we may assume that $z\equiv -w\Mod 3$.
Thus $(z+2w)(z-w)\not\equiv 0\Mod 3$.
So we have $\nu\rat \langle 2,2,3,6\rangle$.

Second, assume that $\nu$ is even.
If we put $\nu=2\nu'$, then $\nu'$ is a positive integer greater than 11 congruent to 2 modulo 3, we have $\nu' \rat \langle 1,1,3,6\rangle$ by \cite[Theorem 2.2]{JO}.
So there is a vector $(x,y,z,w)\in \z^4$ with $xyzw\not\equiv 0\Mod 3$ such that
$$
\nu'=x^2+y^2+3z^2+6w^2.
$$
It follows that
$$
\nu=2x^2+2y^2+3(2w)^2+6z^2
$$
with $xy\cdot 2w\cdot z\not\equiv 0\Mod 3$.
Thus $\nu\rat \langle 2,2,3,6\rangle$.
One may easily check that
$$
3u+13\rat \langle 2,2,3,6\rangle
$$
for every positive integer $u$ with $2\le u\le 863$ except $u=14$.

\noindent {\bf (Case 6)} The sixth case is when $g=p_8(2,2,3,3,3)$.
One may directly check that $p_8(2,2,3,3,3)$ represents all positive integers from 2 to 18 except 14.
On the other hand, every positive integer $u$ greater than or equal to 2 except 14 is represented by $p_8(2,2,3,6)$ by (Case 5) and thus also represented by $p_8(2,2,3,3,3)$.

\noindent {\bf (Case 7)} The seventh case is when $g=p_8(2,3,4,6)$.
Let $v$ be a positive integer greater than 358 and put $\nu=3v+15$.
Then $\nu$ is a positive integer greater than 1089 such that $\nu\equiv 0\Mod 3$.

First, assume that $\nu\equiv 1\Mod 4$.
One may easily check that there is an integer $z_1\in \{1,5,7,11\}$ such that
$$
\nu-9z_1^2\equiv 4,12,20,24,36,40,44,52\ \ \text{or}\ \ 56\Mod{64}.
$$
From this, one may easily check that $\nu-9z_1^2\ra \gen (\langle 2,4,18\rangle)$.
If we put $L(1)=\langle 2,4,18\rangle$, then the class number of $L(1)$ is two and
$$
\gen(L(1))=\left\{ L(1),\ L(2)=\begin{pmatrix}4&2&0\\2&6&2\\0&2&8\end{pmatrix}\right\}.
$$
Note that
$$
L(2)\prec_{4,0}L(1).
$$
Since $\nu-9z_1^2\ra \gen(L(1))$ and $\nu-9z_1^2\equiv 0\Mod 4$, it follows that $\nu-9z_1^2\ra L(1)$.
Thus there is a vector $(x_1,y_1,w_1)\in \z^3$ such that
$$
\nu-9z_1^2=2x_1^2+4y_1^2+18w_1^2.
$$
It follows that
\begin{align*}
\nu&=2x_1^2+4y_1^2+9z_1^2+18w_1^2\\
&=2x_1^2+4y_1^2+3(z_1+2w_1)^2+6(z_1-w_1)^2.
\end{align*}
Note that $2x_1^2+4y_1^2>0$ since $\nu-9z_1^2=18w_1^2$ cannot hold because of the congruence condition on $\nu-9z_1^2$.
So by Lemma \ref{lemJones}, we may assume that $x_1y_1\not\equiv 0\Mod 3$.
If $w\equiv 0\Mod 3$, then we have $z_1+2w_1\equiv z_1-w_1\not\equiv 0\Mod 3$ since $z_1\not\equiv 0\Mod 3$.
When $w\not\equiv 0\Mod 3$, then by changing the sign of $w$, we may take $w_1\equiv -z_1\Mod 3$ so that $z_1+2w_1\equiv z_1-w_1\not\equiv 0\Mod 3$ also.
Thus we have $\nu\rat \langle 2,3,4,6\rangle$.

Second, assume that $\nu\equiv 3\Mod 4$.
One may easily check that there is an integer $y_2\in \{1,5\}$ such that
$$
\nu-3y_2^2\equiv 4,8,12,24\ \ \text{or}\ \ 28\Mod{32}.
$$
Then
$$
\frac{\nu-3y_2^2}{2}\equiv 2,4,6,12\ \ \text{or}\ \ 14\Mod{16}.
$$
By Lemma \ref{lem123}, we have
$$
\frac{\nu-3y_2^2}{2}\rat \langle 1,2,3\rangle.
$$
Thus $\nu-3y_2^2\rat \langle 2,4,6\rangle$ and it follows that $\nu\rat \langle 2,3,4,6\rangle$.

Third, assume that $\nu\equiv 2\Mod 4$.
One may easily check that there is an integer $w_3\in \{2,4\}$ such that
$$
\nu-18w_3^2\equiv 2,6\ \ \text{or}\ \ 10\Mod{16}.
$$
From this, one may easily check that $\nu-18w_3^2\ra \gen(\langle 2,4,9\rangle)$.
If we put $L(3)=\langle 2,4,9\rangle$, then the class number of $L(3)$ is three and
$$
\gen(L(3))=\left\{ L(3), L(4)=\langle 1,2,36\rangle,\ L(5)=\begin{pmatrix}3&1&0\\1&5&2\\0&2&6\end{pmatrix}\right\}.
$$
Note that
$$
L(i)\prec_{24,r}L(3),\quad i=4,5,\ r=0,6,12,18.
$$
Since $\nu-18w_3^2\equiv 0\Mod 6$, we have $\nu-18w_3^2\ra L(3)$.
So there is a vector $(x_3,y_3,z_3)\in \z^3$ such that
\begin{align*}
\nu&=2x_3^2+4y_3^2+9z_3^2+18w_3^2\\
&=2x_3^2+4y_3^2+3(z_3+2w_3)^2+6(z_3-w_3)^2.
\end{align*}
Note that $2x_3^2+4y_3^2>0$ since $\nu-18w_3^2=9z_3^2$ cannot hold because of the congruence condition on $\nu-18w_3^2$.
By Lemma \ref{lemJones}, we may assume that $x_3y_3\not\equiv 0\Mod 3$.
We have $z_3+2w_3\equiv z_3-w_3\not\equiv 0\Mod 3$ when $w_3\equiv 0\Mod 3$.
If $w_3\not\equiv 0\Mod 3$, then by changing the sign of $w$ if necessary, we may take $w_3\equiv -z_3\Mod 3$ so that
$$
z_3+2w_3\equiv z_3-w_3\not\equiv 0\Mod 3.
$$
Thus we have $\nu\rat \langle 2,3,4,6\rangle$.

Last, assume that $\nu\equiv 0\Mod 4$.
We may write $\nu=4^st$ with $s$ a nonnegative integer and $t$ a positive integer not divisible by 4.
Note that $t\equiv 0\Mod 3$ since $\nu\equiv 0\Mod 3$.
Suppose that $t\ge 15$.
Then we have $t\rat \langle 2,3,4,6\rangle$.
So there is a vector $(x_4,y_4,z_4,w_4)\in \z^4$ such that $\nu=2x_4^2+3y_4^2+4z_4^2+6w_4^2$.
It follows that
$$
\nu=4^st=2(2^sx_4)^2+3(2^sy_4)^2+4(2^sz_4)^2+6(2^sw_4)^2,
$$
and we are done.
Thus we may assume that $t\in \{3,6,9\}$.
However we already checked that, for any $l\in \{4^2\cdot 3,4\cdot 6,4\cdot 9\}$, we have
$$
l\rat \langle 2,3,4,6\rangle.
$$
From this follows that $\nu\rat \langle 2,3,4,6\rangle$.
One may directly check that 
$$
3u+15\rat \langle 2,3,4,6\rangle
$$
for every positive integer $u$ with $2\le u\le 358$ except $u=18$.

\noindent {\bf (Case 8)} The eighth case is when $g=p_8(3,4,5,6,10)$.
Let $v$ be a positive integer greater than 384 and put $\nu=3v+28$.
Then $\nu$ is a positive integer greater than 1180 such that $\nu\equiv 1\Mod 3$.

First, assume that $\nu\equiv 0\Mod 2$.
One may easily check that there is an integer $x_1\in \{1,2,4\}$ such that $\nu-10x_1^2$ is an integer greater than 24 congruent to 6 modulo 18.
By Lemma \ref{lem3456}, we have $\nu-10x_1^2\rat \langle 3,4,5,6\rangle$.
Since $x_1\not\equiv 0\Mod 3$, we have $\nu\rat \langle 3,4,5,6,10\rangle$.

Second, assume that $\nu\equiv 3,5,7\Mod 8$.
One may easily check that there is an integer $x_2\in \{1,2\}$ such that $\nu-10x_2^2$ in an integer greater than 99 congruent to 3 or 5 modulo 8.
By Lemma \ref{lem3456}, we have $\nu-10x_2^2\rat \langle 3,4,5,6\rangle$, and thus
$$
\nu\rat \langle 3,4,5,6,10\rangle.
$$

Last, assume that $\nu\equiv 1\Mod 8$.
One easily check that there is an integer $y\in \{1,2,5,7,10,14\}$ such that $\nu-3-6y^2$ is a positive integer satisfying
$$
\nu-3-6y^2\equiv \pm 1\Mod 5,\ \nu-3-6y^2\equiv 2,8,10,14\Mod{16}.
$$
Define a ternary $\z$-lattice $L$ by $L=\langle 4,5,10\rangle$.
From this, one may easily check that $\nu-3-6y^2\ra \gen(L)$, where
$$
\gen(L)=\{ L,\ M=\langle 1,10,20\rangle \}
$$
Since $\nu-3-6y^2$ is even, we have $\nu-3-6y^2\ra \langle 4,10,20\rangle$ and thus 
$$
\nu-3-6y^2\ra L.
$$
So there is a vector $(z,w,v)\in \z^3$ such that
$$
\nu-3-6y^2=4z^2+5w^2+10v^2.
$$
Note that $5w^2+10v^2>0$ since $\nu-3-6y^2=4z^2$ cannot hold because 
$$
\nu-3-6y^2\equiv 2,8,10,14\Mod{16}.
$$
By Lemma \ref{lemJones}, we may assume that $wv\not\equiv 0\Mod 3$.
Since $\nu-3-6y^2\equiv 1\Mod 3$, it follows that $z\not\equiv 0\Mod 3$.
Since we take $y\not\equiv 0\Mod 3$ also, we have
$$
\nu\rat \langle 3,4,5,6,10\rangle.
$$
One may directly check that
$$
3u+28\rat \langle 3,4,5,6,10\rangle
$$
for every positive integer $u$ with $3\le u\le 384$ except $u=27$.
This completes the proof.
\end{proof}

\section{Proofs of the main results}
Note that the proofs for Theorems \ref{thm3} and \ref{thm4} are quite similar to that of Theorem \ref{thm2}, we only provide the proofs for Theorems \ref{thm2} and \ref{thm5}.

\noindent ({\it Proof of Theorem \ref{thm2}})
We first prove that every octagonal form listed in Table \ref{tablet2} is tight $\mathcal T(2)$-universal.
Let $g$ be any octagonal form in Table \ref{tablet2}.
One may directly check that $g$ does not represent 1 and it does represent all integers from 2 to 28.
Now one may easily check that $g$ represents every integer greater than or equal to 29 by using Lemma \ref{lem2233t} and Proposition \ref{propmany}.

To prove the theorem, we use the algorithm in Section 2 with $n=2$.
Note that $U(1)=U(2)=\emptyset$ and
$$
E(3)=\{(2,2,3),(2,3,4)\}
$$
by Lemma \ref{lembase}.
One may directly check that
$$
\psi(2,2,3)=6,\ \ \psi(2,3,4)=8.
$$
Thus we have
$$
E(4)=\left\{ \begin{array}{c}(2,2,2,3),(2,2,3,3),(2,2,3,4),(2,2,3,6),\\(2,3,3,4),(2,3,4,4),(2,3,4,5),(2,3,4,6),(2,3,4,8)\end{array}\right\}.
$$
One may directly check that
$$
\begin{array}{lll}
\psi(2,2,2,3)=8,&\psi(2,2,3,3)=9,&\psi(2,2,3,6)=14,\\[0.3em]
\psi(2,3,3,4)=11,&\psi(2,3,4,4)=12,&\psi(2,3,4,6)=18,
\end{array}
$$
and the other three vectors in $E(4)$ constitute $U(4)$ so that
$$
U(4)=\{(2,2,3,4),(2,3,4,5),(2,3,4,8)\}.
$$
Since $U(k)=\emptyset$ for $k=1,2,3$, we have $NU(4)=U(4)$.
Oone may easily continue to yield that
$$
\begin{array}{llll}
\vert E(5)\vert=52,&\vert U(5)\vert=49,&\vert NU(5)\vert=39,&\vert A(5)\vert=3,\\[0.3em]
\vert E(6)\vert=30,&\vert U(6)\vert=30,&\vert NU(6)\vert=15,&\vert A(6)\vert=0.
\end{array}
$$
Hence the algorithm stops here by Lemma \ref{lemesc}.
Note that
$$
A(5)=\{(2,2,2,3,3),(2,2,3,3,3),(2,2,3,3,5)\}
$$
and
$$
\psi(2,2,2,3,3)=11,\ \psi(2,2,3,3,3)=14,\ \psi(2,2,3,3,5)=14.
$$
One may check that, for $k=5,6$, the vectors $(a_1,a_2,\dots,a_k)\in NU(k)$ exactly corresponds to
the octagonal forms in Table \ref{tablet2}.
Furthermore, one may easily check that
$$
C(2)=\{2,3,4,6,8,9,11,12,14,18\}.
$$
The last part of the theorem immediately follows from this and Lemma \ref{lemcv}.
This completes the proof. \hfill $\qed$

\vskip 10pt
\noindent ({\it Proof of Theorem \ref{thm5}})
For $n=1,2,\dots$, we let
$$
g_n=p_8(n,n,n+1,n+2,\dots,2n-1)\ \ \text{and}\ \ h_n=p_8(n,n+1,n+2,\dots,2n-1,2n).
$$
We first prove that both $g_n$ and $h_n$ are tight $\mathcal T(n)$-universal for every $n\ge 5$.
Note that $p_8(1,1,3,3)$ is universal (see \cite[Theorem2.1]{JO}).
From this and \cite[Lemma 2.2]{MG}, it follows that 
$g_n$ is tight $\mathcal T(n)$-universal for every $n\ge 7$.
On the other hand, one may easily show that the quinary octagonal form $p_8(1,2,3,3,3)$ is universal by using \cite[Theorem 3.1]{JO}.
From this \cite[Lemma 2.2]{MG}, it follows that $h_n$ is tight $\mathcal T(n)$-universal for every $n\ge 9$.
The tight $\mathcal T(n)$-universalities for
$$
g_n,\ n\in \{5,6\} \ \ \text{and}\ \ h_n,\ n\in \{5,6,7,8\}
$$
are already proved in Proposition \ref{propmany}.

From the algorithm, we have
$$
\vert E(n+1)\vert=2,\ \vert U(n+1)\vert=2,\ \vert NU(n+1)\vert=2,\ \vert A(n+1)\vert=0,
$$
and
$$
C(n)=\{n,n+1,n+2,\dots,2n\}.
$$
\hfill $\qed$

\begin{table}[ht]
\caption{New tight $\mathcal T(2)$-universal octagonal forms $p_8(a_1,a_2,\dots,a_k)$}
\label{tablet2}
\begin{center}
\begin{tabular}{|cccccc|c|}
\Xhline{1pt}
$a_1$&$a_2$&$a_3$&$a_4$&$a_5$&$a_6$&Conditions on $a_k  \ (5\le k\le6)$\\
\hline
2&2&3&4&&&\\
2&3&4&5&&&\\
2&3&4&8&&&\\
\hline
2&2&2&2&3&&\\
2&2&2&3&$a_5$&&$5\le a_5\le 8$, $a_5\neq 7$\\
2&2&3&3&$a_5$&&$6\le a_5\le 9$, $a_5\neq 8$\\
2&2&3&5&6&&\\
2&2&3&6&$a_5$&&$6\le a_5\le 14,\ a_5\neq 13$\\
2&3&3&3&4&&\\
2&3&3&4&$a_5$&&$4\le a_5\le 11$, $a_5\neq 5,8,10$\\
2&3&4&4&$a_5$&&$4\le a_5\le 12,\ a_5\neq 5,8,11$\\
2&3&4&6&$a_5$&&$6\le a_5\le 18,\ a_5\neq 8,17$\\
\hline
2&2&2&3&3&$a_6$&$a_6=3,11$\\
2&2&3&3&3&$a_6$&$3\le a_6\le 14,\ a_6\neq 4,6,7,9,13$\\
2&2&3&3&5&$a_6$&$5\le a_6\le 14,\ a_6\neq 6,7,9,13$\\
\Xhline{1pt}
\end{tabular}
\end{center}
\end{table}

\begin{table}[ht]
\caption{New tight $\mathcal T(3)$-universal octagonal forms $p_8(a_1,a_2,\dots,a_k)$}
\label{tablet3}
\begin{center}
\begin{tabular}{|ccccccc|c|}
\Xhline{1pt}
$a_1$&$a_2$&$a_3$&$a_4$&$a_5$&$a_6$&$a_7$& Conditions on $a_k  \ (5\le k\le7)$\\
\hline
3&3&4&5&$a_5$&&&$6\le a_5\le 10$, $a_5\neq 7,8$\\
3&4&4&5&6&&&\\
3&4&5&6&$a_5$&&&$7\le a_5\le 13$, $a_5\neq 9,10,11$\\
\hline
3&3&3&3&4&$5$&&\\
3&3&3&4&4&5&&\\
3&3&3&4&5&$a_6$&&$5\le a_6\le 13,\ a_6\neq 6,9,10$\\
3&3&4&4&5&$a_6$&&$5\le a_6\le 17,\ a_6\neq 6,9,10,15,16$\\
3&3&4&5&5&$a_6$&&$5\le a_6\le 21,\ a_6\neq 6,9,10,19,20$\\
3&3&4&5&7&$a_6$&&$7\le a_6\le 21,\ a_6\neq 9,10,19,20$\\
3&3&4&5&8&$a_6$&&$8\le a_6\le 21,\ a_6\neq 9,10,19,20$\\
3&3&4&5&11&13&&\\
3&3&4&5&13&14&&\\
3&4&5&5&6&$a_6$&&$6\le a_6\le 22,\ a_6\neq 7,8,12,13,20,21$\\
3&4&5&6&6&$a_6$&&$6\le a_6\le 22,\ a_6\neq 7,8,12,13,20,21$\\
3&4&5&6&9&$a_6$&&$9\le a_6\le 36,\ a_6\neq 12,13,34,35$\\
3&4&5&6&10&$a_6$&&$10\le a_6\le 27,\ a_6\neq 12,13,25,26$\\
3&4&5&6&11&$a_6$&&$11\le a_6\le 27,\ a_6\neq 12,13,25,26$\\
3&4&5&6&14&16&&\\
3&4&5&6&16&17&&\\
\hline
3&3&3&4&5&14&16&\\
3&3&3&4&5&16&17&\\
3&3&4&4&4&4&5&\\
3&3&4&4&4&5&$a_7$&$15\le a_7\le 21$, $a_7\neq 17,19,20$\\
3&4&5&5&5&5&6&\\
3&4&5&5&5&6&$a_7$&$20\le a_7\le 27$, $a_7\neq 22,25,26$\\
\Xhline{1pt}
\end{tabular}
\end{center}
\end{table}

\begin{table}[ht]
\caption{New tight $\mathcal T(4)$-universal octagonal forms $p_8(a_1,a_2,\dots,a_k)$}
\label{tablet4}
\begin{center}
\begin{tabular}{|ccccccc|c|}
\Xhline{1pt}
$a_1$&$a_2$&$a_3$&$a_4$&$a_5$&$a_6$&$a_7$& Conditions on $a_k  \ (6\le k\le7)$\\
\hline
4&5&6&7&8&&&\\
\hline
4&4&4&5&6&7&&\\
4&4&5&6&6&7&&\\
4&4&5&6&7&$a_6$&&$7\le a_6\le 23,\ a_6\neq 8,20,21,22$\\
\hline
4&4&5&5&5&6&7&\\
4&4&5&5&6&7&$a_7$&$20\le a_7\le 28,\ a_7\neq 23,25,26,27$\\
\Xhline{1pt}
\end{tabular}
\end{center}
\end{table}


\end{document}